\documentclass{amsart}
\usepackage{color}

\newtheorem{theorem}{Theorem}
\newtheorem{lemma}{Lemma}
\newtheorem{corollary}{Corollary}
\newtheorem{proposition}{Proposition}
\theoremstyle{conjecture}

\theoremstyle{definition}

\newcommand{\twotwo}[4]{\left(\begin{array}{cc}#1&#2\\&\\#3&#4\end{array}\right)}
\theoremstyle{remark}
\newtheorem*{remark}{Remark}
\theoremstyle{remarks}
\newtheorem*{remarks}{Remarks}
\theoremstyle{example}

\numberwithin{equation}{section}

\begin{document}
\title{Loops in $SL(2,\mathbb C)$ and Root Subgroup Factorization}

\author{Estelle Basor and Doug Pickrell}
\email{ebasor@aimath.org}
\email{pickrell@math.arizona.edu}

\begin{abstract} In previous work we proved that for a $SU(2,\mathbb C)$ valued loop
having the critical degree of smoothness (one half of a derivative in the $L^2$ Sobolev sense),
the following are equivalent: (1) the Toeplitz and shifted Toeplitz operators associated to the loop
are invertible, (2) the loop has a triangular factorization, and (3) the loop has a root subgroup
factorization. For a loop $g$ satisfying these conditions, the Toeplitz determinant
$det(A(g)A(g^{-1}))$ and shifted Toeplitz determinant $det(A_1(g)A_1(g^{-1}))$ factor as products in
root subgroup coordinates. In this paper we observe that, at least in broad outline, there is a relatively
simple generalization to loops having values in $SL(2,\mathbb C)$. The main novel features are that (1) root subgroup coordinates
are now rational functions, i.e. there is an exceptional set, and (2) the non-compactness of $SL(2,\mathbb C)$ entails that
loops are no longer automatically bounded, and this (together with the exceptional set) complicates the analysis at the critical exponent.

\end{abstract}
\maketitle

\setcounter{section}{-1}

\section{Introduction}\label{Introduction}

In \cite{Pi3} and \cite{BP} we initiated the study of root subgroup factorization for loops in the compact group
$SU(2,\mathbb C)$. The purpose of this paper is to show that there is a simple way to extend the broad outline of
this theory to loops in the complex group $SL(2,\mathbb C)$. The basic idea is the following. In the unitary case one is considering a loop which satisfies $kk^*=1$. We replace this with $g(g^{-*})^*=1$ (where $g^{-*}$ is shorthand for $(g^{-1})^*$), and in place of a single condition on $k$, we introduce a pair of conditions on $g$ and $g^{-*}$.

In this introduction we will mainly focus on `the classical case', i.e. the unit disk, which for unitary loops is the context of \cite{Pi3}. In Section \ref{BPcase} we will consider `the general case', in which the disk is replaced by a compact Riemann surface with boundary, which is the context of \cite{BP}.

Let $L_{fin}SL(2,\mathbb C)$ denote the group
consisting of functions $S^1 \to SL(2,\mathbb C)$ having finite Fourier series, with pointwise
multiplication. For example suppose that $\zeta=(\zeta^-,\zeta^+) \in \mathbb C^2$, $\zeta^-\zeta^+\ne 1$, $n\in
\mathbb N$, and choose a square root
$$\mathbf a(\zeta):= (1-\zeta^-\zeta^+)^{-1/2}$$
Then the function
$$ S^1 \to SL(2,\mathbb C):z \to
\mathbf a(\zeta) \left(\begin{matrix} 1&
\zeta^- z^{-n}\\
\zeta^+ z^n&1\end{matrix} \right)$$
is in $L_{fin}SL(2,\mathbb C)$.
It is known that $L_{fin}SL(2,\mathbb C)$ is dense in
$C^{\infty}(S^1,SL(2,\mathbb C))$ (by modifying the proof of Proposition 3.5.3 of \cite{PS}).
Also, if
$f(z)=\sum f_n z^n$, let $f^*(z)=\sum \bar f_n z^{-n}$. If $f \in
H^0(\Delta)$, then $f^* \in H^0(\Delta^*)$, where $\Delta$ is the
open unit disk, $\Delta^*$ is the open unit disk at $\infty$, and
$H^0(U)$ denotes the space of holomorphic functions for a domain
$U$.

\begin{theorem}\label{SU(2)theorem1} Suppose that $g_1 \in L_{fin}SL(2,\mathbb C)$. Consider the following three conditions:

(I.1) $g_1$ is of the form
$$g_1(z)=\left(\begin{matrix} a_1(z)&b_1(z)\\
c_1^*(z)&d_1^*(z)\end{matrix} \right),\quad z\in S^1,$$ where $a_1,b_1,c_1$ and
$d_1$ are polynomials in $z$, and $a_1(0)=d_1^*(\infty)\ne 0$.

(I.2) $g_1$ has a factorization of the form
$$g_1(z)=\mathbf a(\eta_n)\left(\begin{matrix} 1& \eta^+_nz^n\\
\eta^-_nz^{-n}&1\end{matrix} \right)..\mathbf a(\eta_0)\left(\begin{matrix} 1&
\eta^+_0\\
\eta^-_0&1\end{matrix} \right),$$ for some finite subset
$\{\eta_0,..,\eta_n\} \subset \{\eta\in\mathbb C^2:1-\eta^-\eta^+\ne 0\}$.

(I.3) $g_1$ and $g_1^{-*}$ have triangular factorizations of the form
$$\left(\begin{matrix} 1&0\\
\sum_{j=0}^n \bar y_jz^{-j}&1\end{matrix} \right)\left(\begin{matrix} \mathbf a_1&0\\
0&\mathbf a_1^{-1}\end{matrix} \right)\left(\begin{matrix} \alpha_1 (z)&\beta_1 (z)\\
\gamma_1 (z)&\delta_1 (z)\end{matrix} \right),$$ where
the third factor is a polynomial in $z$ which is unipotent upper
triangular at $z=0$, and $\mathbf a_1$ is a nonzero constant.

(I.1) and (I.3) are equivalent. (I.2) implies (I.1) and (I.3). The converse holds generically, in the sense
that the $\eta$ variables are rational functions of the Laurent coefficients
for $b_1/a_1$ and $c_1^*/d_1^*$.

Similarly, for $g_2 \in L_{fin}SL(2,\mathbb C)$, consider the following three conditions:

(II.1) $g_2$ is of the form
$$g_2(z)=\left(\begin{matrix} a^*_2(z)&b^*_2(z)\\
c_2(z)&d_2(z)\end{matrix} \right),\quad z\in S^1,$$ where $a_2,b_2,c_2$ and $d_2$ are polynomials
in $z$, $c_2(0)=b_2(0)=0$, and $a_2^*(\infty)=d_2(0)\ne 0$.

(II.2) $g_2$ has a factorization of the form
$$g_2(z)=\mathbf a(\zeta_n)\left(\begin{matrix} 1&\zeta_n^-z^{-n}\\
\zeta_n^+z^n&1\end{matrix} \right)..\mathbf a(\zeta_1)\left(\begin{matrix} 1&
\zeta_1^-z^{-1}\\
\zeta_1^+z&1\end{matrix} \right),$$ for some finite subset
$\{\zeta_1,..,\zeta_n\} \subset \{\zeta\in\mathbb C^2:1-\zeta^-\zeta^+\ne 0\}$.

(II.3) $g_2$ and $g_2^{-*}$ have triangular factorizations of the form
$$\left(\begin{matrix} 1&\sum_{j=1}^n \bar x_jz^{-j}\\
0&1\end{matrix} \right)\left(\begin{matrix}\mathbf a_2&0\\0&\mathbf a_2^{-1}\end{matrix}\right)\left(\begin{matrix} \alpha_2 (z)&\beta_2 (z)\\
\gamma_2 (z)&\delta_2 (z)\end{matrix} \right), $$ where the third factor is a polynomial in $z$ which is unipotent
upper triangular at $z=0$, and $\mathbf a_2$ is a nonzero constant.

(II.1) and (II.3) are equivalent. (II.2) implies (II.1) and (II.3). The converse holds generically, in the sense
that the $\zeta$ variables are rational functions of the Laurent coefficients for $c_2/d_2$ and $b_2^*/a_2^*$.

\end{theorem}

\begin{remark}\label{remarks1} (a) The two sets of conditions are equivalent; they are intertwined by
the outer involution $\sigma$ of $LSL(2,\mathbb C)$ given by
\begin{equation}\label{outer}\sigma(\left(\begin{matrix}a&b\\c&d\end{matrix}\right))=
\left(\begin{matrix}d&cz^{-1}\\bz&a\end{matrix}\right).\end{equation}

(b) It is easy to see that $g_i$ ($i=1$ or $2$) satisfies the conditions in parts 1 and 2 if and only if $g_i^{-*}$ satisfies the same conditions. Consequently we do not need to mention $g_i^{-*}$ explicitly in parts 1 and 2.

(c) The need to choose square roots for the $\mathbf a$ factors frustrates any attempt to formulate a uniqueness statement for the factorization in part 2. This complication can be avoided by modifying the building blocks of the factorization to not include the $\mathbf a$ factors. But this introduces other problems. This is discussed in Section \ref{concludingcomments} .

(d) Clearly the main novelty is that part 2 can fail (for non-unitary loops). To illustrate this, for $g_2$ as in $(II.2)$, one can calculate (as we will do more systematically in the text) that
$$\frac{c_2}{d_2}=\zeta_1^+z+\zeta_2^+(1-\zeta_1^-\zeta_1^+)z^2+
\left((1-\zeta_1^-\zeta_1^+)(1-\zeta_2^-\zeta_2^+)\zeta_3^+-(1-\zeta_1^-\zeta_1^+)\zeta_1^-(\zeta_2^+)^{2}\right)z^3+...$$
and
$$\frac{b_2^*}{a_2^*}=\zeta_1^-z^{-1}+\zeta_2^-(1-\zeta_1^-\zeta_1^+)z^{-2}+
\left((1-\zeta_1^-\zeta_1^+)(1-\zeta_2^-\zeta_2^+)\zeta_3^- -(1-\zeta_1^-\zeta_1^+)\zeta_1^+(\zeta_2^-)^{2}\right)z^{-3}+...$$
For $g_2$ as in (II.1), we can use these series to inductively and rationally solve for the $\zeta$ coordinates, provided that at each step $1-\zeta_k^-\zeta_k^+\ne 0$ (For a unitary $g_2$, $(1-\zeta_k^-\zeta_k^+)=(1+|\zeta_k^-|^2)>0$, hence there is no obstruction).

An example of a loop as in (II.1) which does not have a root subgroup factorization is
$$g_2=\left(\begin{matrix}1-z^{-1}+z^{-2}&z^{-1}+z^{-3}\\z-2z^2+z^3&1-z+z^2\end{matrix}\right) $$
The associated triangular factorizations are
$$g_2=\left(\begin{matrix}1&z^{-1}+z^{-2}+z^{-3}\\0&1\end{matrix}\right) \left(\begin{matrix}1+z-z^{2}&-z\\z-2z^2+z^3&1-z+z^2\end{matrix}\right) $$
and
$$g_2^{-*}=\left(\begin{matrix}1&z^{-1}+z^{-2}-z^{-3}\\0&1\end{matrix}\right)
\left(\begin{matrix}1+z+z^{2}&-z\\-z-z^3&1-z+z^2\end{matrix}\right) $$

It would be highly desirable to understand the exceptional set in some conceptual way.

\end{remark}

There is a $C^{\infty}$ analogue of Theorem \ref{SU(2)theorem1} which we will state in abbreviated form:

\begin{theorem}\label{smoothSLtheorem1}Suppose that $g_i\in C^{\infty}(S^1,SL(2,\mathbb C))$, $i=1,2$. The statements in Theorem
\ref{SU(2)theorem1} apply to $g_i$, where in parts 1 and 3 the coefficients are holomorphic functions in the disk with smooth boundary values (rather than polynomials), and in part 2  $\{\eta_i\}$ and $\{\zeta_k\}$ are rapidly decreasing
sequences of pairs of complex numbers, the roots $\mathbf a(\eta_n),\mathbf a(\zeta_n)$ are chosen to be near $1$ for large $n$, and the limits
$$g_1(z)=\lim_{n\to\infty}\mathbf a(\eta_n)\left(\begin{matrix} 1&\eta^+_nz^n\\
\eta^-_nz^{-n}&1\end{matrix} \right)..\mathbf
a(\eta_0)\left(\begin{matrix} 1&
\eta^+_0\\
\eta^-_0&1\end{matrix} \right)$$ and
$$g_2(z)=\lim_{n\to\infty}\mathbf a(\zeta_n)\left(\begin{matrix} 1&\zeta^-_nz^{-n}\\
\zeta^+_nz^n&1\end{matrix} \right)..\mathbf
a(\zeta_1)\left(\begin{matrix} 1&
\zeta^-_1z^{-1}\\
\zeta^+_1z&1\end{matrix} \right),$$ exist in
$C^{\infty}(S^1,SL(2,\mathbb C))$.

\end{theorem}

The terminology regarding triangular factorization and Toeplitz
operators in the following theorem is reviewed in Section
\ref{triangularfactorization}.

\begin{theorem}\label{U(2)theorem} (a) Suppose $g\in C^{\infty}(S^1,SL(2,\mathbb C))$. The following are
equivalent:

(i) $g$ and $g^{-*}$ have triangular factorizations $g=lmau$ (see
(\ref{factorization})), where $l$ and $u$ have $C^{\infty}$
boundary values.

(ii) $g$ has a factorization of the form
$$g(z)=g_1^*(z)\left(\begin{matrix} e^{\chi(z)}&0\\
0&e^{-\chi(z)}\end{matrix}\right)g_2(z),$$ where $\chi \in
C^{\infty}(S^1,\mathbb C)$, and $g_1$ and $g_2$ are as in Theorem \ref{smoothSLtheorem1}.

(iii) The Toeplitz operators $A(g)$ and $A(g^{-1})$ (see (\ref{multiplicationop}))
and the shifted Toeplitz operators $A_1(g)$ and $A_1(g^{-1})$ (see the paragraph
following (\ref{matrix})) are invertible.

(b) For a generic $g$ as in (a), in terms of root subgroup coordinates,
$$det(A(g)A(g^{-1}))=\left(\prod_{i=0}^{\infty}\frac{1}{(1-\eta^-_i\eta^+_i)^{i}}\right)\times
\left(\prod_{
j=1}^{\infty}e^{2j\chi_j\chi_{-j}}\right)\times
\left(\prod_{k=1}^{\infty}\frac
{1}{(1-\zeta^-_k\zeta^+_{-k})^{k}}\right)$$ and
$$det(A_1(g)A_1(g^{-1}))=\left(\prod_{i=0}^{\infty}\frac{1}{(1-\eta^-_i\eta^+_i)^{i+1}}\right)\times
\left(\prod_{
j=1}^{\infty}e^{2j\chi_j\chi_{-j}}\right)\times
\left(\prod_{k=1}^{\infty}\frac
{1}{(1-\zeta^-_k\zeta^+_k)^{k-1}}\right)$$
\end{theorem}

\begin{remarks} (a) Again note that in part (a), each of the conditions (i)-(iii) holds for $g$ if and only if the
condition holds for $g^{-*}$.

(b) It is easy to generalize this statement to a loop $g$ with values in $GL(2,\mathbb C)$. The only substantial changes
are that the diagonal factor in (ii) of part (a) now has the form $diag(exp(\chi_1(z)),exp(\chi_2(z))$, and in the
part (b) $exp(2j\chi_j\chi_{-j})$ is replaced by $exp(j((\chi_1)_j(\chi_1)_{-j}+(\chi_2)_j(\chi_2)_{-j}))$.

(c) The loop group $LSL(2,\mathbb C)$ has a universal $\mathbb C^{\times}$ central extension
$$0\to \mathbb C^{\times} \to \widehat L SL(2,\mathbb C) \to LSL(2,\mathbb C) \to 0$$
As explained originally in (section 11.3 of) \cite{PS}, one can interpret $\sigma_0=det(A)$ as the essentially unique holomorphic function on $\widehat L SL(2,\mathbb C)$ of level one which is well-defined on the double coset space
$$H^0(D^*,SL(2,\mathbb C))\backslash \widehat L SL(2,\mathbb C)/H^0(D,SL(2,\mathbb C)) $$
There are product formulas for $\sigma_0$ (and also $\sigma_1=det(A_1)$) as in (b) (this is developed for unitary loops in \cite{PP}).
Observe that $\sigma_i(\widehat g)$ is a holomorphic function of level $1$, $\sigma_i(\widehat g^{-1})$ is a holomorphic function of level $-1$, and $\sigma_i(\widehat g)\sigma_i(\widehat g^{-1})$ is of level $0$, i.e. an ordinary function on the loop group.
\end{remarks}

\subsection{Outline of the paper}

Section \ref{triangularfactorization} is a review of standard facts about
triangular factorization. This is an abbreviated version of Section 1 in \cite{Pi3}.

In Sections \ref{SU(2)case} and \ref{SU(2)caseII}, we prove
Theorems \ref{SU(2)theorem1} and \ref{U(2)theorem}, respectively.
In these two sections we also extend the
equivalences to Holder type function spaces. Unlike the unitary case, it is presently unclear how to formulate
results for the
critical Sobolev space $W^{1/2,L^2}$, and there is the appearance of an exceptional set, because the coordinates
$\eta$ and $\zeta$ are rational.

In section \ref{BPcase} we consider `the general case', i.e. we replace
the disk with a compact Riemann surface with boundary.

In an appendix we briefly discuss a possible alternate formulation which avoids the problem of having
to choose square roots of the factors $(1-\zeta^-\zeta^+)$.  This has the drawback of not directly generalizing the unitary case.
There also still remains an exceptional set.

\subsection{Notation} Sobolev spaces will be denoted by $W^s$, and will
always be understood in the $L^2$ sense. The space of sequences
satisfying $\sum_{n=1}^{\infty} n\vert \zeta_n\vert^2<\infty$ will
be denoted by $w^{1/2}$.

We will use \cite{Pe} as a general reference for Hankel and
Toeplitz operators.

We generally use bold letters for scalars.

\section{Triangular factorization for $LSL(2,\mathbb
C)$}\label{triangularfactorization}

The purpose of this section is to recall some simple facts about triangular factorization. This is an abbreviated
version of Section 1 from \cite{Pi3}, where proofs and further references can be found.

Suppose that $g\in L^1(S^1,SL(2,\mathbb C))$. A triangular
factorization of $g$ is a factorization of the form
\begin{equation}\label{factorization}g=l(g)m(g)a(g)u(g),\end{equation}
where
\[l=\left(\begin{array}{cc}
l_{11}&l_{12}\\
l_{21}&l_{22}\end{array} \right)\in H^0(\Delta^{*},SL(2,{\mathbb
C})),\quad l(\infty )=\left(\begin{array}{cc}
1&0\\
l_{21}(\infty )&1\end{array} \right),\] $l$ has a $L^2$ radial
limit, $m=\left(\begin{array}{cc}
m_0&0\\
0&m_0^{-1}\end{array} \right)$, $m_0\in S^1$,
$a(g)=\left(\begin{array}{cc}
a_0&0\\
0&a_0^{-1}\end{array} \right)$, $a_0>0$,
\[u=\left(\begin{array}{cc}
u_{11}&u_{12}\\
u_{21}&u_{22}\end{array} \right)\in H^0(\Delta ,SL(2,{\mathbb
C})),\quad u(0)=\left(\begin{array}{cc}
1&u_{12}(0)\\
0&1\end{array} \right),\] and $u$ has a $L^2$ radial limit. Note
that (\ref{factorization}) is an equality of measurable functions
on $S^1$. A Birkhoff (or Wiener-Hopf, or Riemann-Hilbert)
factorization is a factorization of the form $g=g_-g_0g_+$, where
$g_-\in H^0(\Delta^*,\infty;SL(2,\mathbb C),1)$, $g_0\in
SL(2,\mathbb C)$, $g_+\in H^0(\Delta,0;SL(2,\mathbb C),1)$, and
$g_{\pm}$ have $L^2$ radial limits on $S^1$. Clearly $g$ has a
triangular factorization if and only if $g$ has a Birkhoff
factorization and $g_0$ has a triangular factorization, in the
usual sense of matrices.

Note that $g$ has a triangular factorization, $g=lmau$, if and only if $g^{*}$ has a triangular factorization,
$g^*=u^*m^*al^*$.
\begin{proposition}Birkhoff and triangular factorizations are
unique. \end{proposition}

As in \cite{PS}, consider the polarized Hilbert space
\begin{equation}\label{polarization}\mathcal H:=L^2(S^1,C^2)={\mathcal H}^{+}\oplus {\mathcal
H}^{-},\end{equation} where $\mathcal H^{+}=P_{+}\mathcal H$
consists of $L^2$-boundary values of functions holomorphic in
$\Delta$. If $g\in L^{\infty}(S^1,SL(2,\mathbb C))$, we write the
bounded multiplication operator defined by $g$ on $\mathcal H$ as
\begin{equation}\label{multiplicationop}M_g=\left(\begin{array}{cc}
A(g)&B(g)\\
C(g)&D(g)\end{array} \right)\end{equation} where
$A(g)=P_{+}M_gP_{+}$ is the (block) Toeplitz operator associated
to $g$ and so on. If $g$ has the Fourier expansion $g=\sum
g_nz^n$,
$g_n=\left(\begin{matrix}a_n&b_n\\c_n&d_n\end{matrix}\right)$,
then relative to the basis for $\mathcal H$:
\begin{equation}\label{basis} ..
\epsilon_1z,\epsilon_2z,\epsilon_1,\epsilon_2,\epsilon_1z^{-1},\epsilon_2z^{-1},..\end{equation}
where $\{\epsilon_1,\epsilon_2\}$ is the standard basis for
$\mathbb C^2$, the matrix of $M_g$ is block periodic of the form
\begin{equation}\label{matrix}\begin{array}{ccccccccc}&.&.&.&.&.&.&.&\\..&a_0&b_0&a_1&b_1&\vert&a_2&b_2&..
\\..&c_0&d_0&c_1&d_1&\vert&c_{2}&d_{2}&..\\
..&a_{-1}&b_{-1}&a_0&b_0&\vert&a_1&b_1&.. \\
..&c_{-1}&d_{-1}&c_0&d_0&\vert&c_1&d_1&..\\-&-&-&-&-&-&-&-&-\\
..&a_{-2}&b_{-2}&a_{-1}&b_{-1}&\vert&a_0&b_0&..\\
..&c_{-2}&d_{-2}&c_{-1}&d_{-1}&\vert&c_0&d_0&..\\&.&.&.&.&.&.&.&
\end{array}\end{equation}
From this matrix form, it is clear that, up to equivalence, $M_g$
has just two types of ``principal minors", the matrix representing
$A(g)$, and the matrix representing the shifted Toeplitz operator
$A_1(g)$, the compression of $M_g$ to the subspace spanned by
$\{\epsilon_iz^j:i=1,2,j>0\}\cup\{\epsilon_1\}$. Relative to the
basis (\ref{basis}), the involution $\sigma$ defined by
(\ref{outer}) is equivalent to conjugation by the shift operator,
i.e. the matrix of $M_{\sigma(g)}$ is obtained from the matrix for
$M_g$ by shifting one unit along the diagonal (in either
direction: the result is the same, because $M_g$ commutes with
$M_z$, the square of the shift operator). Consequently the shifted
Toeplitz operator is equivalent to the operator $A(\sigma(g))$.

\begin{theorem}\label{factorizationthm} Suppose that $g\in L^{\infty}(S^1,SL(2,\mathbb C))$.

(a) If $A(g)$ is invertible, then $g$ has a Birkhoff
factorization, where
\begin{equation}\label{factorformula}(g_0g_+)^{-1}=[A(g)^{-1}\left(\begin{matrix}1\\0\end{matrix}\right),
A(g)^{-1}\left(\begin{matrix}0\\1\end{matrix}\right)].\end{equation}

(b) If $A(g)$ and $A_1(g)$ are invertible, then $g$ has a
triangular factorization.

\end{theorem}

In Theorem \ref{factorizationthm} we are assuming that $g$ is
bounded. It is not generally true that the factors $g_{\pm}$ are
bounded. Recall (see \cite{CG}) that a Banach $*$-algebra $\mathbb
A \subset L^{\infty}(S^1)$ is said to be decomposing if
$$\mathbb A = \mathbb A_+ \oplus \mathbb A_-,$$ i.e. $P_+:\mathbb A
\to \mathbb A_+$ is continuous. For example $C^s(S^1)$, provided $s>0$ and nonintegral (see page 60 of
\cite{CG}), $W^s(S^1)$, provided $s>1/2$ (Note: $W^{1/2}$ is not an algebra),
and $\mathcal{B}$, the intersection of the Wiener algebra and $W^{1/2}$, i.e. the set of functions
$\phi$ such that the Fourier coefficients satisfy
 \[\sum_{k =-\infty} ^{\infty} |\phi_{k}| +
 \Big(\sum_{k = -\infty}^{\infty} |k|\cdot|\phi_{k}|^{2}\Big)^{1/2} < \infty,
\] are all decomposing algebras.

\begin{corollary}\label{decomposing} Suppose that $g \in L^{\infty}(S^1,SL(2,\mathbb
C))$ belongs to a decomposing algebra $\mathbb A$ and has a
Birkhoff factorization. Then the factors $g_{\pm}$ belong to
$\mathbb A$.
\end{corollary}

This follows from the continuity of $P_+$ on $\mathbb A$ and the
formula in (a) of Theorem \ref{factorizationthm}.

\begin{theorem}\label{hartmann}If $g \in L^{\infty}(S^1,SL(2,\mathbb
C))$, then $B(g)$ and $C(g)$ are compact operators if and only if
$g\in VMO$, the space of functions with vanishing mean
oscillation. If $g\in QC:=L^{\infty}\cap VMO$, then $A(g)$ and
$D(g)$ are Fredholm of index $0$.
\end{theorem}

The first statement is due to Hartmann, and the second to Douglas
(see pages 27 and 108 of \cite{Pe}, respectively).

\section{Proof of Theorem \ref{SU(2)theorem1}, and Generalizations
to Other Function Spaces}\label{SU(2)case}

\subsection{Equivalence of parts 1 and 3}

We first prove a Holder space refinement of the equivalence of parts 1 and 3
in Theorems \ref{SU(2)theorem1} and \ref{smoothSLtheorem1}.

\begin{theorem}\label{SU(2)theorem1smooth} Suppose that $g_1 \in C^s(S^1,SL(2,\mathbb C))$,
where $s>0$ and nonintegral. The following are equivalent:

(I.1) $g_1$ is of the form
$$g_1(z)=\left(\begin{matrix} a_1(z)&b_1(z)\\
c_1^*(z)&d_1^*(z)\end{matrix} \right),\quad z\in S^1,$$ where $a_1,b_1,c_1,d_1\in
H^0(\Delta)$ have $C^s$ boundary values, $a_1(0)=d_1^*(\infty)\ne 0$, and the pairs $a_1$ and $b_1$, and $c_1$ and $d_1$,
do not simultaneously vanish at a point in $\Delta$.

(I.3) $g_1$ and $g_1^{-*}$ have triangular factorizations of the form
$$g_1=\left(\begin{matrix} 1&0\\
\sum_{j=0}^{\infty}y^*_jz^{-j}&1\end{matrix} \right)\left(\begin{matrix} \mathbf a_1&0\\
0&\mathbf a_1^{-1}\end{matrix} \right)\left(\begin{matrix} \alpha_1 (z)&\beta_1 (z)\\
\gamma_1 (z)&\delta_1 (z)\end{matrix} \right)$$
and
$$g_1^{-*}=\left(\begin{matrix} 1&0\\
\sum_{j=0}^{\infty}\mathbf y^*_jz^{-j}&1\end{matrix} \right)\left(\begin{matrix} \mathbf a_1^*&0\\
0&(\mathbf a_1^*)^{-1}\end{matrix} \right)\left(\begin{matrix} \alpha'_1 (z)&\beta'_1 (z)\\
\gamma'_1 (z)&\delta'_1 (z)\end{matrix} \right),$$ where the factors
have $C^s$ boundary values, and the third factors are unipotent upper triangular at $z=0$.

Similarly, the following are equivalent:

(II.1) $g_2$ is of the form
$$g_2(z)=\left(\begin{matrix} a_2^{*}(z)&b_2^{*}(z)\\
c_2(z)&d_2(z)\end{matrix} \right),\quad z\in S^1,$$ where $a_2,b_2,c_2,d_2\in
H^0(\Delta)$ have $C^s$ boundary values, $b_2(0)=c_2(0)=0$, $a_2^*(\infty)=d_2(0)\ne 0$, and the pairs $a_2$ and $b_2$, and
$c_2$ and $d_2$, do not simultaneously vanish at a point in $\Delta$.

(II.3) $g_2$ and $g_2^{-*}$ have triangular factorizations of the form
$$g_2=\left(\begin{matrix} 1&\sum_{j=1}^{\infty}x^*_jz^{-j}\\
0&1\end{matrix} \right)\left(\begin{matrix}\mathbf a_2&0\\
0& \mathbf a_2^{-1}\end{matrix} \right)\left(\begin{matrix} \alpha_2 (z)&\beta_2 (z)\\
\gamma_2 (z)&\delta_2 (z)\end{matrix} \right) $$
and
$$g_2^{-*}=\left(\begin{matrix} 1&\sum_{j=1}^{\infty}\mathbf x^*_jz^{-j}\\
0&1\end{matrix} \right)\left(\begin{matrix}\mathbf a_2^*&0\\
0& (\mathbf a_2^*)^{-1}\end{matrix} \right)\left(\begin{matrix} \alpha'_2 (z)&\beta'_2 (z)\\
\gamma'_2 (z)&\delta'_2 (z)\end{matrix} \right), $$
where all the factors have $C^s$ boundary values, and the third factors are unipotent upper triangular at $z=0$.

\end{theorem}

\begin{remark} When $g_2\in L_{fin}SL(2,\mathbb C)$, the determinant condition
$a_2^*d_2-b_2^*c_2=1$ can be interpreted as an equality of finite Laurent
expansions in $\mathbb C\setminus\{0\}$. Together with $d_2(0)\ne 0$, this implies
that $c_2$ and $d_2$ do not simultaneously vanish. Thus the added
hypotheses in (I.1) and (II.1) of Theorem
\ref{SU(2)theorem1smooth} are superfluous in the finite case.

\end{remark}

\begin{proof} As we remarked in the Introduction, the two sets of
conditions are intertwined by the outer involution $\sigma$. We will consider the second set (For the sake of variety, we will consider the first set in the proof of Theorem \ref{SU(2)theorem4smooth} below).

It is straightforward to check that $(II.3)\implies (II.1)$. By multiplying the
matrices in the first displayed line in $(II.3)$, we see that $c_2=\mathbf a_2^{-1}\gamma_2$ and
$d_2=\mathbf a_2^{-1}\delta_2$, and these cannot simultaneously vanish at a
point in $\Delta$. Similarly by multiplying the matrices in the second displayed line in $(II.3)$, we see that
$a_2$ and $b_2$ are holomorphic and cannot simultaneously vanish. The unipotence of the third factors ensures that
$b_2(0)=c_2(0)=0$, and $a_2^*(\infty)=d_2(0)=\mathbf a_2^{-1}\ne 0$.

Now suppose that we are given a loop $g_2$ satisfying the
conditions in (II.1). We first claim that the Toeplitz operator for $g_2$ is invertible (For
this argument, we only need to assume that $g_2\in L^{\infty}\cap VMO$ (i.e. quasicontinuous)).
Suppose that
$$A(g_2)f=P_+(\left(\begin{matrix}a_2^*&b_{2}^*\\c_2&d_2\end{matrix}\right)
\left(\begin{matrix}f_1\\f_2\end{matrix}\right))
=\left(\begin{matrix}0\\0\end{matrix}\right).$$ Then
$c_2f_1+d_2f_2=0\in H^0(\Delta)$, and hence by the independence of $c_2$
and $d_2$ around $S^1$, $(f_1,f_2)=\lambda(d_2,-c_2)$. Because $c_2$ and
$d_2$ do not simultaneously vanish, this implies that $\lambda$ is
holomorphic in $\Delta$. We also have $(a_2^*\lambda
d_2+b_2^*\lambda(-c_2))_+=\lambda_+=0$ (here we have used the fact that $g_2$ has values in $SL(2,\mathbb C)$).
Thus $\lambda=0$. The assumption that $g_2\in C^s$ implies that $g_2$ is quasicontinuous, hence the Toeplitz operator is
Fredholm (see Theorem \ref{hartmann}).  Thus the nonvanishing of the kernel implies that the
Toeplitz operator is invertible [Note: conversely, if $c_2$ and $d_2$
have a common zero $z_0\in \Delta$, then the Toeplitz operator is
not invertible: take $\lambda=1/(z-z_0)$]. The same argument shows
that $A_1(g_2)$, and also $D(g_2)$, are invertible.

We must now show that this loop has a triangular factorization as
in (II.3), i.e. we must solve for $\mathbf a_2$, $x^*$, and so on, in
\begin{equation}\label{solve}g_2(z)=\left(\begin{matrix} a_2^{*}(z)&b_2^{*}(z)\\
c_2(z)&d_2(z)\end{matrix} \right)=\left(\begin{matrix} 1&\sum_{j=1}^n \bar x_jz^{-j}\\
0&1\end{matrix} \right)\left(\begin{matrix} \mathbf a_2&0\\
0&\mathbf a_2^{-1}\end{matrix} \right)\left(\begin{matrix} \alpha_2 (z)&\beta_2 (z)\\
\gamma_2 (z)&\delta_2 (z)\end{matrix} \right).\end{equation} The
form of the second row implies that we must have $\mathbf a_2=d_2(0)^{-1}$ ,
and
\begin{equation}\label{eqn0}\gamma_2=\mathbf a_2 c_2, \quad \text{and} \quad \delta_2=\mathbf a_2 d_2\end{equation}
because $\delta_2(0)=1$. This does define $\mathbf a_2\ne 0$, $\gamma_2$ and
$\delta_2$ in a way which is consistent with $(II.3)$, because
$c_2(0)=0$ and $d_2(0)\ne 0$.

Using (\ref{eqn0}), the first row in (\ref{solve}) is equivalent
to
\begin{equation}\label{eqn1}a_2^{*}=\mathbf a_2\alpha_2+x^{*}c_2,\quad \text{and}\quad
b_2^{*}=\mathbf a_2\beta_2+x^{*}d_2\end{equation} In the finite case, by
considering the second equation as an equality in $\mathbb C\setminus\{0\}$,
we can immediately obtain that $x^*=(b_2^*/d_2)_-$. The $C^s$ case is
more involved (in this case (\ref{eqn1}) is just an equality on the circle).

Consider the Hardy space polarization
$$H:=L^2(S^1,d\theta )=H^{+}\oplus H^{-},$$
and the operator
$$T:H^{-}\to H^{-}\oplus H^{-}:x^{*}\to (((c_2x^{*})_{-},(d_2x^{*})_{
-}).$$ The operator $T$ is the restriction of $D(g_2)^*=D(g_2^*)$
to the subspace $\{(x^*,0)\in H^-\}$, consequently it is
injective with closed image.

The adjoint of $T$ is given by
$$T^{*}:H^{-}\oplus H^{-}\to H^{-}:(f^{*},g^{*})\to c_2^{*}f^{*}+d_2^{
*}g^{*}.$$ If $(f^{*},g^{*})\in ker(T^{*})$, then
$c_2^{*}f^{*}+d_2^{*}g^{*}$ vanishes in the closure of $\Delta^{*}$ (in an a.e. sense on the boundary).
This implies
$(f^{*},g^{*})=\lambda^{*}(d_2^{*},-c_2^{*})$, where $\lambda^{*}$ is
holomorphic in $ \Delta^{*}$ (because $d_2^*$ and $c_2^*$ do not simultaneously vanish)
and vanishes at $\infty$ (because
$d_2^{*}(\infty )$ equals the conjugate of $d_2(0)$, which is nonzero). We now claim that $((a_2^{*})_{-},b_2^{*})\in
ker(T^{*})^{\perp}$: $$\int ((a_2^{*})_{-}f+b_2^{*})g)d\theta
=\int\lambda (a_2^{*}d_2-b_2^{*} c_2)d\theta =\int\lambda d\theta =0,$$
because $\lambda (0)=0$. Because $T$ has closed image, there
exists $x^{*}\in H^{-}$ such that
\begin{equation}\label{eqn2}(a_2^{*})_{-}=(x^{*}c_2)_{-},\quad and\quad
b_2^{*}=(x^{*}d_2)_{-}.\end{equation} We can now solve for $\alpha_2$
and $\beta_2$ in (\ref{eqn1}), because $\mathbf a_2=d_2(0)^{-1}$ is already determined.

This shows that $g_2$ in $(II.1)$
has a triangular factorization as in $(II.3)$. When $g_2\in C^s$,
by Corollary \ref{decomposing}, the factors are $C^s$.

We can similarly obtain a triangular factorization for $g_2^{-*}$, because
$$g^{-*}=\left(\begin{matrix}d_2^*&-c_2^*\\-b_2&a_2\end{matrix}\right)$$
satisfies the same hypotheses. This
completes the proof of Theorem \ref{SU(2)theorem1smooth}.
\end{proof}

In the preceding theorem, the hypotheses are symmetric in terms of $g_i$ and $g_i^{-*}$. It is possible to break this symmetry
and modify Theorem \ref{SU(2)theorem1smooth} to characterize $g_i$ which have triangular factorizations as in part 3.

\begin{theorem}\label{SU(2)theorem4smooth} Suppose that $g_1 \in C^s(S^1,SL(2,\mathbb C))$,
where $s>0$ and non-integral. The following are equivalent:

(I.1) $g_1$ is of the form
$$g_1(z)=\left(\begin{matrix} a_1(z)&b_1(z)\\
\tilde c_1(z)&\tilde d_1(z)\end{matrix} \right),\quad z\in S^1,$$ where $a_1,b_1\in
H^0(\Delta)$ have $C^s$ boundary values, $a_1(0)\ne 0$, and the pair $a_1$ and $b_1$
do not simultaneously vanish at a point in $\Delta$ (To emphasize, no conditions are imposed on
$\tilde c_1$ and $\tilde d_1$ beyond the fact that they are $C^s$ on $S^1$).

(I.3) $g_1$ has triangular factorization of the form
$$g_1=\left(\begin{matrix} 1&0\\
\sum_{j=0}^{\infty}y^*_jz^{-j}&1\end{matrix} \right)\left(\begin{matrix} \mathbf a_1&0\\
0&\mathbf a_1^{-1}\end{matrix} \right)\left(\begin{matrix} \alpha_1 (z)&\beta_1 (z)\\
\gamma_1 (z)&\delta_1 (z)\end{matrix} \right)$$
where the factors have $C^s$ boundary values.

Similarly, the following are equivalent:

(II.1) $g_2$ is of the form
$$g_2(z)=\left(\begin{matrix} \tilde a_2(z)&\tilde b_2(z)\\
c_2(z)&d_2(z)\end{matrix} \right),\quad z\in S^1,$$ where $c_2,d_2\in
H^0(\Delta)$ have $C^s$ boundary values, $c_2(0)=0$, $d_2(0)\ne 0$, and the pair
$c_2$ and $d_2$ does not simultaneously vanish at a point in $\Delta$.

(II.3) $g_2$ has triangular factorization of the form
$$g_2=\left(\begin{matrix} 1&\sum_{j=1}^{\infty}x^*_jz^{-j}\\
0&1\end{matrix} \right)\left(\begin{matrix}\mathbf a_2&0\\
0& \mathbf a_2^{-1}\end{matrix} \right)\left(\begin{matrix} \alpha_2 (z)&\beta_2 (z)\\
\gamma_2 (z)&\delta_2 (z)\end{matrix} \right) $$
where the factors have $C^s$ boundary values.

\end{theorem}

\begin{proof} We consider the equivalence of the first set of conditions. It is evident that (I.3) implies (I.1).
Suppose that $g_1$ is as in (I.1). As in the second paragraph of the preceding proof, we know that $g_1$ has
a triangular factorization. We will write the factorization as $g_1=lU$, where $U$ includes the diagonal. Since $C^s$ is a decomposing algebra, the factors are $C^s$. Let
$$l= \twotwo{h_{1}}{h_{2}}{h_{3}}{h_{4}} \text{ and  } U=\twotwo{k_{1}}{k_{2}}{k_{3}}{k_{4}} , $$
The important point is that $l$  ($U$) has values in $SL(2,\mathbb C)$ in $\Delta^*$ ($\Delta$, respectively).
This implies that

\[ h_{1} k_{1} + h_{2}k_{3} = a_{1}, \,\,\,\,\, h_{1} k_{2} + h_{2}k_{4} = b_{1}\]
and
\[ h_{1} k_{1}k_{4} + h_{2}k_{3}k_{4} = a_{1} k_{4}, \,\,\,\,\, h_{1} k_{2}k_{3} + h_{2}k_{4}k_{3} = b_{1}k_{3} \]

Subtracting and using the fact that $k_{1}k_{4} - k_{2}k_{3} = 1$, we have that $$h_{1} = a_{1} k_{4} - b_{1}k_{3} .$$
But this says that $h_{1}$ is the $C^s$ boundary values of a function which are holomorphic in $\Delta$, and of a function which is holomorphic in $\Delta^*$. Therefore $h_1$ is constant. The same argument also says that $h_{2}$ is a constant. The normalization of $l$ at $z=\infty$ implies that $h_{1}=1$ and $h_2=0$. Thus the triangular factorization of $g_1$ has the form claimed in (I.3).

\end{proof}

\subsection{Proof Of Theorem \ref{SU(2)theorem1} and \ref{smoothSLtheorem1}}

It follows from Theorem \ref{SU(2)theorem1smooth} that (II.1) and (II.3) are equivalent.
It is straightforward to calculate that a loop as in (II.2) has
the matrix form in (II.1); see Lemma \ref{keylemma} below for the form of the expressions one obtains
for the matrix entries of $g_2(\zeta)$, as functions of $\zeta$.

The main point is to show that (II.1) implies (II.2), in a generic sense. We first want to explain
what we mean by generic in a concrete way. Assume that $g_2$ is as in (II.1), where the degrees of $a_2,b_2,c_2$
and $d_2$ are $\le n$ and $n>0$. We first claim that $a_2$ and $d_2$ automatically have degree $<n$.
To see this we use the determinant condition $a_2^*d_2-b_2^*c_2=1$. For the left hand side, the coefficient of $z^{-n}$ is
$$(a_2^*)_{-n}(d_2)_0-(b_2^*)_{-n}(c_2)_0=(a_2^*)_{-n}(d_2)_0$$
(using $c_2(0)=0$). Since $n>0$, this coefficient must vanish. Since $d_2(0)\ne 0$, this implies $a_2$ has degree $<n$.
In a similar way, by considering the coefficient of $z^n$ of the left hand side,  we see that $(d_2)_n=0$.

We now do a dimension count. There are $2n$ $\zeta_k^{\pm}$ parameters. There are $4n-1$
coefficients for $a_2^*,b_2^*,c_2$ and $d_2$ (we subtract $1$ because $a_2(0)^*=d_2(0)$). The pointwise determinant of $g_2$
is a Laurent polynomial with coefficients strictly between $-n$ and $n$ (this uses the special form of $g_2$), hence the condition on this determinant introduces $2n-1$ constraints. Thus the number of $\zeta$ parameters agrees with the number of free coefficient parameters.

The coefficients are clearly polynomial functions of the $\zeta$ parameters, and we will now show how to express the
$\zeta$ variables as rational functions of the coefficients. One way to do this (in the special finite case we are presently considering) is
the following. For
$\zeta_n^+=(c_2)_n$ and $\zeta_n^-=(b_2^*)_{-n}$,
\begin{equation}\label{reduction}\left(\begin{matrix} 1&-\zeta^-_nz^{-n}\\
-\zeta^+_nz^n&1\end{matrix} \right)\left(\begin{matrix} a^*_2(z)&b^*_2(z)\\
c_2(z)&d_2(z)\end{matrix} \right)=\left(\begin{matrix} a^*_2(z)-\zeta_n^-c_2(z)z^{-n}&b^*_2(z)-\zeta_n^-d_2(z)z^{-n}\\
c_2(z)-\zeta_n^+a_2^*(z)z^n&d_2(z)-\zeta_n^+b_2^*(z)z^n\end{matrix} \right)\end{equation}
has (Laurent polynomial) entries of coefficients with  $-n<degree<n$. If $1-\zeta_n^-\zeta_n^+\ne 0$ (a generic condition), then we can reduce the order, and hence by an induction argument we can assert that generically $g_2$ can be written as in (II.2). A drawback of this argument is that it applies exclusively to finite type loops.\\

To complete the proofs of Theorems \ref{SU(2)theorem1} and \ref{smoothSLtheorem1}, it remains to show that the $\eta$ and $\zeta$ coefficients can be expressed as explicit rational functions of the Taylor coefficients of $c_2/d_2$ and $b_2^*/a_2^*$.

The following statement is completely algebraic.

\begin{theorem}\label{solvingsubgpcoords} (a) For $g_1(\eta)$ as in (II.1), the meromorphic function $b_1/a_1$
has Taylor series $\sum_{n=0}^{\infty}\psi_nz^n$
where $\psi_n$ is the sum of terms
$$\psi_n= (-1)^r(\eta^+_{i_{0}})\left(\eta^-_{j_1}(\eta^+_{i_1})\right)...\left(\eta^-_{j_r}(\eta^+_{i_r})\right) $$
where $j_s<i_s$ and $j_s\le i_{s-1}$ for $s=1,..,r$, and $\sum_{s=1}^{r+1} i_s -\sum_{s=1}^r j_s=n$; in particular
$$\psi_n=(\eta^+_n)\prod_{s=1}^{n-1}(1-\eta_s^-\eta_s^+)+polynomial(\eta^-_s,\eta^+_s,s<n)$$

(b) For $g_2(\zeta)$, the meromorphic function $c_2/d_2$
has Taylor series $\sum_{n=1}^{\infty}\xi_nz^n$
where $\xi_n$ is the sum of terms
$$(-1)^r(\zeta^+_{i_{0}})\left(\zeta^-_{j_1}(\zeta^+_{i_1})\right)...\left(\zeta^-_{j_r}(\zeta^+_{i_r})\right) $$
where $j_s<i_s$ and $j_s\le i_{s-1}$ for $s=1,..,r$, and $\sum_{s=1}^{r+1} i_s -\sum_{s=1}^r j_s=n$; in particular
$$\xi_n=(\zeta^+_n)\prod_{s=1}^{n-1}(1-\zeta_s^-\zeta_s^+)+polynomial(\zeta^-_s,\zeta_s^+,s<n)$$
\end{theorem}

For example
$$\frac{c_2}{d_2}=\zeta_1^+z+\zeta_2^+(1-\zeta_1^-\zeta_1^+)z^2+
(1-\zeta_1^-\zeta_1^+)((1-\zeta_2^-\zeta_2^+)\zeta_3^+-\zeta_1^-(\zeta_2^+)^{2})z^3+...$$
In a similar way, by considering $g_2^{-*}$,
$$\frac{b_2^*}{a_2^*}=\zeta_1^-z^{-1}+\zeta_2^-(1-\zeta_1^-\zeta_1^+)z^{-2}+
(1-\zeta_1^-\zeta_1^+)((1-\zeta_2^-\zeta_2^+)\zeta_3^- -\zeta_1^+(\zeta_2^-)^{2})z^{-3}+...$$
From this it is evident that we can recursively and rationally solve for the $\zeta$ coordinates in terms of the Taylor coefficients for
$b_2^*/a_2^*$ and $c_2/d_2$.

\begin{proof} We will prove part (b); part (a) is proven in the same way.

Let $\xi=\gamma_2/\delta_2$. Our strategy is completely straightforward: we will first recall
the formulas for the coefficients of $\gamma_2$ and $\delta_2$, and we will then calculate the Taylor series for the quotient.
The following (which is completely straightforward) is from \cite{Pi3}, which is reminiscent of the Pauli exclusion principle:

\begin{lemma}\label{keylemma} If
\begin{equation}\label{2product}\left(\begin{matrix}\alpha_2^*(z)&\beta_2^*(z)\\
\gamma_2(z)&\delta_2(z)\end{matrix}\right)=\lim_{n\to\infty}\mathbf a(\zeta_n)\left(\begin{matrix} 1&\zeta^-_nz^{-n}\\
\zeta^+_nz^n&1\end{matrix} \right)..\mathbf
a(\zeta_1)\left(\begin{matrix} 1&
\zeta^-_1z^{-1}\\
\zeta^+_1z&1\end{matrix} \right),\end{equation}
then
$$\gamma_2(z)=\sum_{n=1}^{\infty}\gamma_{2,n}z^n,$$
$$\gamma_{2,n}=\sum (\zeta^+_{i_1})\zeta^-_{j_1}...(\zeta^+_{
i_r})\zeta^-_{j_r}(\zeta^+_{i_{r+1}}),$$ where the sum is over
multiindices satisfying
$$0<i_1<j_1<..<j_r<i_{r+1},\quad\sum i_{*}-\sum j_{*}=n,$$
and
$$\delta_2(z)=1+\sum_{n=1}^{\infty}\delta_{2,n}z^n,$$
$$\delta_{2,n}=\sum\zeta^-_{j_1}(\zeta^+_{i_1})...\zeta^-_{j_r}(
\zeta^+_{i_r}),$$ where the sum is over multiindices satisfying
$$0<j_1<i_1<..<i_r,\quad\sum (i_{*}-j_{*})=n$$
\end{lemma}

For example
$$\gamma_2=\zeta_1^+z+(\zeta_2^++(\zeta_1^+(\zeta_2^-\zeta_3^++\zeta_3^-\zeta_4^+...)z^2+... $$
(with the exception of the first, all the coefficients involve infinite sums) and
$$\delta_2=1+(\zeta_1^-\zeta_2^++\zeta_2^-\zeta_3^+...)z+...$$
(all the coefficients involve infinite sums).

We now want to prove (b) of Theorem \ref{solvingsubgpcoords}.
To simplify notation, let $\gamma:=\gamma_2$, and write $\delta_2:=1+\delta$.
Then $\gamma_2/\delta_2=\gamma(1-\delta+\delta^2-..$, and the nth coefficient of $\gamma_2/\delta_2$ equals
$$\gamma_n-(\gamma\delta)_n+(\gamma\delta^2)_n-..+(-1)^{n-1}(\gamma\delta^{n-1})_n$$
Each of the terms in this sum, according to the Lemma, has an expression as an infinite sum. According to the statement
of the theorem, all but finitely many of these terms cancel out.

To explain in a leisurely way how this comes about, first consider
\begin{equation}\label{sumofterms}(\gamma\delta)_n=\sum_{k=1}^{n-1}\gamma_{k}\delta_{n-k}\end{equation}
This is a sum of terms of the form
$$\left(\zeta^+_{i_1})\zeta^-_{j_1}...(\zeta^+_{
i_r})\zeta^-_{j_r}(\zeta^+_{i_{r+1}})\right)\left(\zeta^-_{j'_1}(\zeta^+_{i'_1})...\zeta^-_{j'_{r'}}(
\zeta^+_{i'_{r'}})\right)$$
where
$$0<i_1<j_1<..<j_r<i_{r+1}, \qquad \sum i_{*}-\sum j_{*}=k, $$
$$0<j'_1<i'_1<..<i'_{r'}, \qquad  \sum i'_{*}-\sum j'_{*}=n-k$$
If $i_{r+1}<j'_1$, then this product will exactly cancel with a term in the corresponding sum for $\gamma_n$; it is in some sense
obeying a Pauli exclusion principle.  The only term in the sum for $\gamma_n$ that is not canceled is $\zeta^+_n$; this is the only term which cannot be broken into two terms as in the sum for $\gamma\delta$. If $j'_1\le i_{r+1}$, then we keep this term; however,
we will see that many of these terms are canceled by subsequent terms appearing in the sum (\ref{sumofterms}).

Consider $(\gamma\delta^s)_n$. This is a sum of terms of the form
\begin{equation}\label{term1}\left(\zeta^+_{i_1})\zeta^-_{j_1}...(\zeta^+_{
i_r})\zeta^-_{j_r}(\zeta^+_{i_{r+1}}\right) \times
\left(\zeta^-_{j^1_1}(\zeta^+_{i^1_1})...\zeta^-_{j^1_{r^1}}(
\zeta^+_{i^1_{r^1}})\right)\times ... \times \left(\zeta^-_{j^s_1}(\zeta^+_{i^s_1})...\zeta^-_{j^s_{r^s}}(
\zeta^+_{i^s_{r^s}})\right)\end{equation}
where each of the factors separated by $\times$ (the first factor comes
from $\gamma$, and the other $s$ factors come from $\delta^s$) satisfy the appropriate constraints in the Lemma (in particular the sum
of the $i$ indices minus the sum of the $j$ indices equals $n$).

At one extreme, it may happen that all of the indices in (\ref{term1}) are increasing, i.e. $i_{r+1}<j^1_1$ and $i^{s'}_{r^{s'}}<j^{s'+1}_1$ for $s'=1,..,s-1$. By removing some of the $\times$, we see that this product will have occurred in
all of the preceding terms $(\gamma\delta^{s'})_n$, $s'=0,..,s-1$ (which occur with alternating signs). Similarly if
$r>0$ or $r^j>1$ for some $j$, then we can insert $\times$ and this more finely factored product will occur in some subsequent terms $(\gamma\delta^{s'})_n$, for $s<s'$; the largest such $s'$ is $S=r+r^1+..+r^s$, in which for each factor
of $\delta$, the corresponding factor
$$\zeta^-_{j_1}(\zeta^+_{i_1})...\zeta^-_{j_{r'}}(
\zeta^+_{i_{r'}})$$ is irreducible in the sense that it cannot be split into a product of two similar factors, i.e. $r'=1$.
The number of times the product (\ref{term1}) occurs in one of the terms in (\ref{sumofterms}) depends on the number of ways we can insert $\times$. Taking into account
the signs that occur in (\ref{sumofterms}), the coefficient of this product in (\ref{sumofterms}) is
$$\sum_{j=0}^S(-1)^j\left(\begin{matrix}S\\j\end{matrix}\right)=0$$ Thus this
product completely cancels out in the sum (\ref{sumofterms}).

At the opposite extreme, the term (\ref{term1}) may have the same form as in the statement of the theorem, i.e. $r=0$, and
$r^j=1$, $j=1,..,s$. In this case this term occurs in exactly one of the terms $(\gamma\delta^{s'})_n$. There is no cancelation.

In between these two extremes, we are considering a product for which, in the expression (\ref{term1}), there is a positive number of instances when $j^{s''}_{1}\le i^{s''-1}_{r^{s''-1}}$ for some $s''=1,..,s-1$. As in the first extreme case, we can possibly remove some of the $\times$ to
see that this term occurs in earlier terms $(\gamma\delta^{s'})_n$ ($s'<s$), and we can possibly insert some $\times$ to see that it occurs in some
later terms $(\gamma\delta^{s'})_n$ ($s<s'$). For definiteness we can suppose that $s$ is as large as possible, i.e. that $r=0$
and each $r^j=1$, so that it is just a question of removing $\times$. If $s_0$ is the smallest $s'$ such that the product occurs
in $\gamma\delta^{s'}$, then $s_0<s$ (because we are not in the second extreme case) and the coefficient of this product in (\ref{sumofterms}) is
$$\sum_{j=s_0}^s(-1)^j\left(\begin{matrix}s-s_0\\j\end{matrix}\right)=0$$
Thus this product completely cancels out.

If we multiply out $(\zeta^+_n)\prod(1+\zeta_i\overline \zeta_i)$, then we see that each of the terms does occur in
the sum in part (b). This proves the last claim in part (b).

\end{proof}

As we have previously explained, this completes the proofs of Theorems \ref{SU(2)theorem1} and \ref{SU(2)theorem1smooth}).

\subsection{Toeplitz Determinants}

The following theorem follows by analytic continuation from the unitary case (the first equalities are due to
Widom, see \cite{W}).

\begin{theorem}\label{det1} For $g_i$ as in Theorem \ref{smoothSLtheorem1} having a root subgroup factorization,
$det(A(g_1)A(g_1^{-1}))$ equals
$$\lim_{N\to{\infty}}det(A_N(g_1))=det(1-B(g_1)C(g_1^{-1}))
=\prod_{n\ge 1}(1-\eta_n^-\eta^+_n)^{-n}$$ and $det(A(g_2)A(g_2^{-1}))$ equals
$$\lim_{N\to{\infty}}det(A_N(g_2))=det(1-B(g_2)C(g_2^{-1}))=\prod_{n\ge 1}(1-\zeta_n^-\zeta^+_n)^{-n},$$
where $A_N$ denotes the finite dimensional compression of $A$ to
the span of $\{\epsilon_i z^k:0\le k\le N\}$, and in the third
expressions, $x$ and $y$ are viewed as multiplication operators on
$H=L^2(S^1)$, with Hardy space polarization.
\end{theorem}

\section{Proof of Theorem \ref{U(2)theorem}, and
Generalizations}\label{SU(2)caseII}

In the process of proving Theorem \ref{U(2)theorem}, we will also prove the
following Holder version of the result:

\begin{theorem}\label{maintheorem}Assume $s>0$ and nonintegral, or
$s=\infty$. For $g\in C^s(S^1,SL(2,\mathbb C))$, the following are
equivalent:

(i) $g$ has a triangular factorization $g=lmau$, where $l$ and $u$
have $C^s$ boundary values, and similarly for $g^{-*}$

(ii) $g$ has a factorization $g=g_1^*\lambda g_2$, where
$g_1,g_2\in C^s(S^1,SL(2,\mathbb C))$ satisfy the equivalent conditions
(I.1) and (I.3) ((II.1) and (II.3), respectively) of Theorem
\ref{SU(2)theorem1smooth}, and $\lambda\in C^s(S^1,\mathbb C\setminus\{0\})_0$, the identity component.

\end{theorem}

\begin{proof}We will use the notation in (\ref{factorization}) for $g$,
and the notation in Theorem \ref{SU(2)theorem1smooth} for the
entries of the $g_i$ and their triangular factorizations. Without
much comment, we will use the fact that $C^s$ is a decomposing
algebra, so that factors in various decompositions will remain in
$C^s$.

We first show that (ii) implies (i). Suppose that $g\in
C^s(S^1,SL(2))$ can be factored as
$g=g_1^*\left(\begin{matrix}\lambda&0\\0&\lambda^{-1}
\end{matrix}\right) g_2$, as in (ii). Write
$\lambda=exp(-\chi_-+\chi_0+\chi_+)$, where $\chi_0 \in \mathbb C$
and $\chi_+\in H^0(\Delta)$, $\chi_+(0)=0$, with $C^s$ boundary
values. Then $g$ has triangular factorization of the form
\begin{equation}g=l(g)\left(\begin{matrix}e^{\chi_0}\mathbf a_1\mathbf a_2&0\\0&(e^{\chi_0}\mathbf a_1\mathbf a_2)^{-1}\end{matrix}\right)u(g),\end{equation}
where $m_0=e^{\chi_0}\in S^1$, $a_0=a_1a_2>0$,
\begin{equation}\label{lmatrix}l(g):=\left(\begin{matrix}l_{11}&l_{12}\\l_{21}&l_{22}\end{matrix}\right)
=\left(\begin{matrix} \alpha_1^{*}&\gamma_1^{*}\\
\beta_1^{*}&\delta_1^{*}\end{matrix}
\right)\left(\begin{matrix}e^{-\chi^{*}}&0\\0&e^{\chi^*}\end{matrix}\right)\left(\begin{matrix}
1&a_1^2e^{2\chi_0}P_-(ye^{2\chi^{*}}+x^*e^{2\chi})\\
0&1\end{matrix} \right)\end{equation} and
\begin{equation}\label{umatrix}u(g):=\left(\begin{matrix}u_{11}&u_{12}\\u_{21}&u_{22}\end{matrix}\right)
= \left(\begin{matrix}
1&a_2^{-2}e^{-2\chi_0}P_+(ye^{2\chi^*}+x^*e^{2\chi})\\0&1\end{matrix}\right)
\left(\begin{matrix}e^{\chi}&0\\0&e^{-\chi}\end{matrix}\right)
\left(\begin{matrix} \alpha_2&\beta_2\\
\gamma_2&\delta_2\end{matrix}\right)\end{equation}

Since
$$g^{-*}(z)=(g_1^{-*})^*(z)\left(\begin{matrix} e^{-\chi^*(z)}&0\\
0&e^{\chi^*(z)}\end{matrix}\right)g_2^{-*}(z),$$ where
$$g_1^{-*}(z)=\left(\begin{matrix} d_1(z)&-c_1(z)\\
-b_1^*(z)&a_1^*(z)\end{matrix} \right),\quad g_2^{-*}(z)=\left(\begin{matrix} d^*_2(z)&-c^*_2(z)\\
-b_2(z)&a_2(z)\end{matrix} \right),\quad z\in S^1,$$
have precisely the same form as $g_1$ and $g_2$, $g^{-*}$ has a triangular factorization. To make the notation manageable,
we will write
$l(g^{-*})=l'$ and so on, for the factors corresponding to $g^{-*}$.
Notation aside, we see that (ii) implies (i).

To explain how the converse is proved, we will first show how $g_1,g_2$ and $\chi$ are recovered from the
triangular factorizations for $g$ and $g^{-*}$.

From the triangular factorizations of $g$ and $g^{-*}$ we obtain four relatively simple pairs of equations.
The first two come from considering the first columns of $l(g)$ and $l(g^{-*})$:
$$\left(\begin{matrix}l_{11}\\l_{21}\end{matrix}\right)=e^{\chi_-}\left(\begin{matrix}\alpha_1^*\\\beta_1^*\end{matrix}\right), \quad \left(\begin{matrix}l_{11}(g^{-1})\\l_{21}(g^{-*})\end{matrix}\right)=e^{-\chi_+^*}\left(\begin{matrix}(\alpha_1')^*\\(\beta_1')^*\end{matrix}\right) $$
The second two come from considering the second rows of $u(g)$ and $u(g^{-*})$:
$$(u_{21},u_{22})=e^{-\chi_+}(\gamma_2,\delta_2), \quad (u_{21}(g^{-*}),u_{22}(g^{-*}))=e^{\chi_-^*}(\gamma_2',\delta_2')$$
The first two column equations imply
$$l_{11}l_{11}'-l_{21}l_{21}'=e^{-\chi_++\chi_-^*}(\alpha_1^*\alpha_1'-\beta_1^*\beta_1')$$
$$=e^{-\chi_++\chi_-^*}(\mathbf a_1^*\mathbf a_1')^{-1}(a_1^*d_1-b_1^*c_1)=e^{-\chi_++\chi_-^*}(\mathbf a_1^*\mathbf a_1')^{-1}$$

Similarly the second two row equations imply
$$u_{21}u_{21}'-u_{22}u_{22}'=e^{-\chi_++\chi_-}(\gamma_2(\gamma_2')^*-\delta_2(\delta_2')^*)$$
$$=e^{-\chi_++\chi_-^*}\mathbf a_2(\mathbf a_2')^* (a_2^*d_2-b_2^*c_2)=e^{-\chi_++\chi_-^*}\mathbf a_2(\mathbf a_2')^*$$
These two equations determine $\chi_{\pm}$, and the products $\mathbf a_1^*\mathbf a_1'$ and $\mathbf a_2(\mathbf a_2')^*$,
provided that the consistency condition
$$\mathbf a_1^*\mathbf a_1'(l_{11}l_{11}'-l_{21}l_{21}')=(\mathbf a_2(\mathbf a_2')^*)^{-1}(u_{21}u_{21}'-u_{22}u_{22}')$$
is satisfied. We also obtain
$$a_1=\mathbf a_1^* e^{-\chi_-^*}l_{11}^*, b_1=\mathbf a_1^* e^{-\chi_-^*}l_{21}^*$$
$$c_1=-(\mathbf a_1')^* e^{\chi_+}l_{21}^*, d_1=(\mathbf a_1')^* e^{\chi_+}l_{11}^* $$
$$ a_2=(\mathbf a_2')^{-1} e^{-\chi_-^*}u_{22}',b_2=-(\mathbf a_2')^{-1} e^{-\chi_-^*}u_{21}'$$
and
$$c_2=(\mathbf a_2)^{-1} e^{\chi_+}u_{21}, d_2=(\mathbf a_2)^{-1} e^{\chi_+}u_{22} $$

To check that the consistency condition is satisfied one spells out the equation $g(g^{-*})^*=1$ in terms of the
LDU factorizations. This is basically the same as in the unitary case, and we will not write out the details.

\end{proof}

\begin{remarks} (a) In a generic situation there is a direct way to find the $\eta$ and $\zeta$ factors which bypasses
having to find $\chi$. For example
$$u_{21}(g)/u_{22}(g)=c_2/d_2=\zeta^+_1 z+\zeta^+_2(1-\zeta^-_1\zeta^+_1)z^2$$
$$+\left(\zeta^+_3(1-\zeta^-_1\zeta^+_1)(1-\zeta^-_2\zeta^+_2)+
\zeta^-_1(\zeta^+_2)^2(1-\zeta^-_1\zeta^-_1)\right)z^3+...$$
and
$$u_{21}(g^{-*})/u_{22}(g^{-*})=-b_2/a_2=\zeta^-_1 z^{-1}+\zeta^-_2(1-\zeta^-_1\zeta^+_1)z^{-2}$$
$$+\left(\zeta^-_3(1-\zeta^-_1\zeta^+_1)(1-\zeta^-_2\zeta^+_2)+
\zeta^+-_1(\zeta^-_2)^2(1-\zeta^-_1\zeta^-_1)\right)z^{-3}+...$$
and we can solve for the $\zeta$ variables recursively. Similarly
$$l_{21}(g)/l_{11}(g)=b_1^*/a_1^*\quad  \text{  and } \quad l_{21}(g^{-*})/l_{11}(g^{-*})=-c_1^*/c_1^*$$
and we can solve for the $\eta $ variables recursively.

(b) The arguments we have given show that the $\eta$ and $\zeta$ factors are uniquely determined. However it is still necessary to
choose the square roots for the $\mathbf a$ factors. As a consequence the $\mathbf a_1$ and $\mathbf a_2$ factors are not uniquely determined, unlike the unitary case (compare with \cite{Pi3}).

\end{remarks}

\subsection{Toeplitz Determinants}

\begin{theorem}
Suppose that $g\in C^s(S^1,SL(2,\mathbb C))$, $s>1/2$, and $g$ has a triangular factorization. Then
\begin{equation}\label{det2}det(A(g)A(g^{-1}))=det(A(g_1)A(g_1^{-1}))det(A(\lambda)A(\lambda^{-1}))det(A(g_2A(g_2^{-1}))\end{equation}
If $g_1$ and $g_2$ have root subgroup factorizations, then
$$det(A(g)A(g^{-1}))=\left(\prod_{i=1}^{\infty}(1-\eta_i^-\eta^+_i)^{-i}\right)^*exp(2\sum_{j=1}^{\infty}j
\chi_j\chi_{-j})\prod_{k=1}^{\infty}(1-\zeta_k^-\zeta_k^+)^{-k}.$$
\end{theorem}

\section{Riemann Surfaces and Factorization}\label{BPcase}

In this section we will note how the statements in \cite{BP} can be generalized to $SL(2,\mathbb C)$ valued loops.
The proofs in that paper are easily modified in the same ways as we have indicated above.

Suppose that $\Sigma$ is a connected compact Riemann surface with nonempty
boundary $S$ (a disjoint union of circles). Let $\widehat{\Sigma}$
denote the double,
$$\widehat{\Sigma}=\Sigma^* \circ \Sigma,$$
where $\Sigma^*$ is the adjoint of $\Sigma$, i.e. the surface $\Sigma$
with the orientation reversed, and the composition is sewing along
the common boundary $S$. Let $R$ denote the antiholomorphic
involution (or reflection) fixing $S$.

In previous sections we have considered `the classical case' $\Sigma=D$, the closed unit disk. In this case
$S=S^1$, $\widehat{\Sigma}$ is isomorphic to the Riemann sphere, and (in this
realization)
$R(z)=1/z^*$, where $z^*=\overline{z}$, the complex conjugate.
This example has the exceptional feature that there is a
large automorphism group, $PSU(1,1)$, acting by linear
fractional transformations.

We now choose a basepoint, denoted by $(0)$, in the
interior of $\Sigma$, and we let $(\infty)$ denote the reflected basepoint for
$\Sigma^*$. In the classical case, without loss of generality because of the $PSU(1,1)$ symmetry,
we can assume the basepoint is $z=0$. Given the data $(\Sigma,(0))$, following ideas
of Krichever and Novikov, a reasonable function on $S$ has a `linear triangular factorization'
\begin{equation}\label{LinearTF}f=f_-+f_0+f_+\end{equation}
where $f_{\pm}$ is holomorphic in the interior of $\Sigma$ ($\Sigma^*$, respectively), with appropriate boundary
behavior, depending on the smoothness of $f$, $f_{+}((0))=0$, $f_{-}((\infty))=0$, and $f_0$ is the restriction
to $S$ of a meromorphic function which
belongs to a $\text{genus}(\widehat{\Sigma})+1$ dimensional complementary subspace, which we refer to
as the vector space of zero modes (see Proposition 2.3 of \cite{BP}).
In the classical case $f_0$ is the zero mode for the Fourier series of $f$.

A holomorphic map $\mathfrak z:\widehat{\Sigma}\to \widehat{D}$ is said to be strictly equivariant
if it satisfies
\begin{equation}\label{dfnequivfn}\mathfrak z(R(q))=\frac1{\mathfrak z(q)^*}\end{equation}
and maps $\Sigma$ to $D$ (and hence
$\Sigma^*$ to $D^*$). When we refer to the classical case ($\Sigma=D$),
it will be understood that $\mathfrak z(z)=z$. For a function $f:U\subset \widehat{\Sigma}\to \mathcal L(\mathbb C^N)$, define
$f^*(q)=f(R(q))^*$, where $(\cdot)^*$ is the Hermitian adjoint. If $f\in H^0(\Sigma)$ (i.e. a holomorphic function in
some open neighborhood of $\Sigma$), then $f^*\in H^0(\Sigma^*)$. If $q\in S$, then $f^*(q)=f(q)^*$, the ordinary
complex conjugate of $f(q)$.

\begin{theorem}\label{SU(2)theorem11} Suppose that $g_1 \in C^{\infty}(S,SL(2,\mathbb C))$. Consider the following
three conditions:

(I.1) $g_1$ is of the form
$$g_1(z)=\left(\begin{matrix} a_1(z)&b_1(z)\\
c_1^*(z)&d_1^*(z)\end{matrix} \right),\quad z\in S,$$ where $a_1,b_1,c_1$ and
$d_1$ are boundary values of holomorphic function in $\Sigma$ with $a_1((0))$ and $d_1((0))$ nonzero,
and the pairs $a_1$ and $b_1$, and $c_1$ and $d_1$,
do not simultaneously vanish at a point in $\Sigma$.

(I.2) $g_1$ has a `root subgroup factorization' of the form
$$g_1(z)=\lim_{n\to\infty}\mathbf a(\eta_n)\left(\begin{matrix} 1&\eta^+_n\mathfrak z^n\\
\eta_n\mathfrak z^{-n}&1\end{matrix} \right)..\mathbf
a(\eta_0)\left(\begin{matrix} 1&
\eta^+_0\\
\eta_0&1\end{matrix} \right),$$ for some rapidly decreasing sequence
$\{\eta_0,..,\eta_n,..\}$ of pairs of complex numbers, and for some strictly equivariant
function $\mathfrak z$ with $\mathfrak z((0))=0$.

(I.3) $g_1$ and $g_1^{-*}$ have (multiplicative triangular) factorizations of the form
$$\left(\begin{matrix} 1&0\\ y^*(z)+y_0(z)&1\end{matrix} \right)\left(\begin{matrix} \mathbf a_1&0\\
0&\mathbf a_1^{-1}\end{matrix} \right)\left(\begin{matrix} \alpha_1 (z)&\beta_1 (z)\\
\gamma_1 (z)&\delta_1 (z)\end{matrix} \right),$$ where $\mathbf a_1\ne 0$,
the third factor is a $SL(2,\mathbb C)$-valued holomorphic function in $\Sigma$ which is unipotent upper
triangular at $0$, $y=y_+$ is holomorphic in $\Sigma$, and $y_0$ is as in (\ref{LinearTF}).

Then (I.2) implies (I.1) and (I.3) (with $y_0=0$),
and (I.1) and (I.3) are equivalent.

Similarly, consider $g_2 \in C^{\infty}(S,SU(2))$ and the following statements:

(II.1) $g_2$ is of the form
$$g_2(z)=\left(\begin{matrix} a_2^{*}(z)&b_2^{*}(z)\\
c_2(z)&d_2(z)\end{matrix} \right),\quad z\in S^1,$$ where $a_2,b_2,c_2$ and $d_2$
are boundary values of holomorphic functions in $\Sigma$, $b_2((0))=c_2((0))=0$, and $d_2((0))\ne 0$,
and the pairs $a_2$ and $b_2$, and $c_2$ and $d_2$,
do not simultaneously vanish at a point in $\Sigma$.

(II.2) $g_2$ has a `root subgroup factorization' of the form
$$g_2(z)=\lim_{n\to\infty}\mathbf a(\zeta_n)\left(\begin{matrix} 1&\zeta_n\mathfrak z^{-n}\\
\zeta^+_n\mathfrak z^n&1\end{matrix} \right)..\mathbf
a(\zeta_1)\left(\begin{matrix} 1&
\zeta_1\mathfrak z^{-1}\\
\zeta^+_1\mathfrak z&1\end{matrix} \right),$$ for some rapidly decreasing sequence
$\{\zeta_1,..,\zeta_n,..\}$ of complex numbers, and for some strictly equivariant $\mathfrak z$ with
$\mathfrak z((0))=0$.

(II.3) $g_2$ and $g_2^{-*}$ have (triangular) factorizations of the form
$$\left(\begin{matrix} 1& x^*(z)+x_0(z)\\
0&1\end{matrix} \right)\left(\begin{matrix} a_2&0\\
0&a_2^{-1}\end{matrix} \right)\left(\begin{matrix} \alpha_2 (z)&\beta_2 (z)\\
\gamma_2 (z)&\delta_2 (z)\end{matrix} \right), $$ where $\mathbf a_2\ne 0$, the third factor is a
$SL(2,\mathbb C)$-valued holomorphic function which is unipotent
upper triangular at $(0)$, $x=x_+$ is holomorphic
in $\Sigma$, and $x_0$ is as in (\ref{LinearTF}).

Then (II.2) implies (II.1) and (II.3) (with $x_0=0$), and (II.1) and (II.3) are equivalent.

\end{theorem}

\begin{theorem}\label{SU(2)theorem12} Suppose $g\in C^{\infty}(S,SL(2,\mathbb C))$. If $g$ has a factorization
$$g(z)=g_1^*(z)\left(\begin{matrix} e^{\chi(z)}&0\\
0&e^{-\chi(z)}\end{matrix}\right)g_2(z),$$ where $\chi \in
C^{\infty}(S,\mathbb C)$, and $g_1$ and $g_2$ are as in (I.1) and (II.1), respectively, of Theorem \ref{SU(2)theorem1},
then $E(g)$, the holomorphic $SL(2,\mathbb C)$ bundle on $\widehat{\Sigma}$
defined by $g$ as a transition function, is semi-stable.

\end{theorem}

\begin{remark}It is very unlikely that there exists a converse to this statement, because it appears that there is a more
general condition on a multi-loop $g$ also implies that $E(g)$ is semi-stable; see Theorem 4.3 of \cite{BP}. We will not include it here because it is more difficult to state, we lack specific examples, and we have not established the converse.
\end{remark}

\subsection{Spin Toeplitz Operators}

Assume that $\Sigma$ has a spin structure. There is an induced spin structure for $\widehat{\Sigma}$ which has an
anti-holomorphic reflection symmetry compatible with $R$. We additionally assume that the $\overline \partial$
operator for spinors on $\widehat \Sigma$ is invertible. In this case there is a (pre-)Hilbert space polarization
for the space of ($\mathbb C^2$ valued) spinors along $S$,
$$\Omega^{1/2}(S)\otimes\mathbb C^2=H^{1/2}(\Sigma)\otimes\mathbb C^2\oplus H^{1/2}(\Sigma^*)\otimes\mathbb C^2$$
where $H^{1/2}(\Sigma)$ denotes the space of holomorphic spinors on $\Sigma$. Given a (measurable) loop $g:S\to SU(2)$, there is an
associated unitary multiplication operator $M_g$ on $\Omega^{1/2}(S)\otimes\mathbb C^2$, and relative to the polarization
$$M_g=\left(\begin{matrix}A(g)&B(g)\\C(g)&D(g)\end{matrix}\right)$$
In the classical case $A(g)$ ($B(g)$) is the classical block Toeplitz operator (Hankel operator, respectively),
associated to the symbol $g$. In general the `spin Toeplitz (Hankel) operators' $A(g)$ ($B(g)$, respectively) have
many of the same qualitative properties
as in the classical case, because the projection $\Omega^{1/2}\to H^{1/2}$ differs from the classical projection
by a smoothing operator.

The next result follows by analytic continuation from the unitary case.

\begin{theorem}\label{SU(2)theorem13} Suppose that $g:S\to SL(2,\mathbb C)$ (is smooth and) has a factorization as in Theorem \ref{SU(2)theorem12}.
Then for any choice of spin structure for which $\overline{\partial}$ is invertible,
 $$det(A(g)A(g^{-1}))=det(A(g_1)A(g_1^{-1}))det(\dot A(e^{\chi}) \dot A(e^{-\chi}))^2det(A(g_2)A(g_2^{-1}))$$
where in the middle factor $\dot A(e^{\chi})$ is the compression to $H^{1/2}(\Sigma)$ of $e^{\chi}$
as a (scalar) multiplication operator on $\Omega^{1/2}(S)$ (which accounts for the square of this factor).

\end{theorem}

The proof of the following result is essentially the same as in the unitary case, see Theorem 5.3 of \cite{BP}. We include the proof, because we bungled the last line in the proof of Theorem 5.3 of \cite{BP}.

\begin{theorem}\label{SU(2)theorem14} Suppose that $g_1,g_2$ are $L^{\infty}$ multi-loops $S\to SL(2,\mathbb C)$ of the form
$$g_1=\left(\begin{matrix} a_1(z)&b_1(z)\\
\tilde c_1(z)&\tilde d_1(z)\end{matrix} \right)\text{ and } g_2=\left(\begin{matrix} \tilde a_2(z)&\tilde b_2(z)\\
c_2(z)&d_2(z)\end{matrix} \right)
$$
where $a_1,b_1,c_2$ and
$d_2$ are boundary values of holomorphic function in $\Sigma$ (No additional conditions are imposed on $\tilde c_1,\tilde d_1,\tilde a_2,$ and $\tilde b_2$). Then $$A(g_1^*g_2)=A(g_1^*)A(g_2)$$.
\end{theorem}

\begin{proof} Because $$A(g_1^*g_2)=A(g_1^*)A(g_2)+B(g_1^*)C(g_2)$$ this equivalent to showing that
$B(g_1^*)C(g_2)=0$. We will prove this by direct calculation.
Suppose that $f:=\left(\begin{matrix}f_1\\f_2\end{matrix}\right)\in H_+=H^{1/2}(\Delta)$. Then
$$B(g_1^*)C(g_2)f=P_+\left(g_1^*P_-(g_2f)\right)=P_+\left(g_1^*\left(\begin{matrix}P_-(\tilde a_2 f_1-\tilde b_2 f_2)\\0\end{matrix}\right) \right)$$
$$=P_+(\left(\begin{matrix}a_1^*P_-(\tilde a_2 f_1-\tilde b_2 f_2)\\b_1^*P_-(\tilde a_2 f_1-\tilde b_2 f_2)\end{matrix}\right))=0$$
Thus $B(g_1^*)C(g_2)=0$.
\end{proof}

\section{Concluding Comments}\label{concludingcomments}

\subsection{An Alternate Formulation}

In this paper the basic building blocks (complementing diagonal loops) are loops of the form
$$ S^1 \to SL(2,\mathbb C):z \to
\mathbf a(\zeta) \left(\begin{matrix} 1&
\zeta^- z^{-n}\\
\zeta^+ z^n&1\end{matrix} \right)$$
where $\mathbf a(\zeta)=(1-\zeta^-\zeta^+)^{-1/2}$ is a choice of square root.
For $n>0$ this has triangular factorization
$$ \left(\begin{matrix} 1&
\zeta^- z^{-n}\\0&1\end{matrix} \right) \left(\begin{matrix} \mathbf a(\zeta)&0\\
0&\mathbf a(\zeta)^{-1}\end{matrix} \right) \left(\begin{matrix} 1&0\\
\zeta^+ z^n&1\end{matrix} \right)$$
Furthermore the inverse has exactly the same form, with $-(\zeta^+,\zeta^-)$ in place of $(\zeta^-,\zeta^+)$.

The set
$$\{\zeta \in \mathbb C^2: 1-\zeta^-\zeta^+\ne 0\}$$
is connected, but it is not simply connected.  There is a preferred choice of the square root on the `unitary' slice
$\zeta^+=-(\zeta^+)^* $ (the conjugate), namely
$$(1-\zeta^-\zeta^+)^{1/2}=(1+|\zeta^-|^2)^{1/2}>0$$
But this does not uniquely determine the choice of square root by analytic continuation,
and the basic loop above is actually parameterized by a point in a double covering of the set of parameters. This makes the parameter space quite complicated.

An alternative is to remove the square root ambiguity by considering the loop with triangular decomposition
$$ \left(\begin{matrix} 1+\zeta^-\zeta^+&
\zeta^- z^{-n}\\\zeta^+ z^{n}&1\end{matrix} \right) =\left(\begin{matrix} 1&
\zeta^- z^{-n}\\&1\end{matrix} \right) \left(\begin{matrix} 1&\\
\zeta^+ z^n&1\end{matrix} \right)$$
This has the feature that the diagonal is the identity. It has the drawback that when we consider $(\cdot)^{-*}$ of this loop,
the resulting loop has triangular factorization
$$ \left(\begin{matrix} 1+\zeta^-\zeta^+&
\zeta^- z^{-n}\\\zeta^+ z^{n}&1\end{matrix} \right)^{-*} =\left(\begin{matrix} 1&
(-\zeta^+)^* z^{-n}\\&1\end{matrix} \right)\left(\begin{matrix}(1+\zeta^-\zeta^+)^{-*}&\\
&(1+\zeta^-\zeta^+)^{*}\end{matrix} \right) \left(\begin{matrix} 1&\\
(-\zeta^-)^* z^n&1\end{matrix} \right)$$
if and only if $1+\zeta^-\zeta^+\ne 0$.

Let $N^+$ denote the group of loops $u\in H^0(D,SL(2,\mathbb C))$ such that $u(0)$ is unipotent upper triangular,
$N^-$ the group of loops $l\in H^0(D^*,SL(2,\mathbb C))$ such that $l(\infty)$ is unipotent lower triangular, and $c^{\infty}\otimes \mathbb C^2$
the Frechet space of $\mathbb C^2$ valued sequences which are rapidly decreasing.

Suppose that $\eta_{*+1}\in c^{\infty}\otimes \mathbb C^2$, $\chi\in C^{\infty}(S^1,\mathbb C)$, and $\zeta\in c^{\infty}\otimes \mathbb C^2$. It is easy to see that the product of limits
\begin{equation}\label{forwardmap}g=\left(\prod_{0\le i<\infty}^{\leftarrow} \left(\begin{matrix} 1+\eta_i^-\eta_i^+&
\eta_i^+ z^{n}\\\zeta_i^- z^{-n}&1\end{matrix} \right)\right)^*\left(\begin{matrix} exp(\chi)&0\\0&exp(-\chi)\end{matrix} \right)\left(\prod_{1\le k<\infty}^{\leftarrow}\left(\begin{matrix} 1+\zeta_k^-\zeta_k^+&
\zeta_k^- z^{-n}\\\zeta_k^+ z^{n}&1\end{matrix} \right) \right)\end{equation}
exists in $C^{\infty}(S^1,SL(2,\mathbb C))$

\begin{theorem} (a) $\eta,\chi,\zeta \to g$ defines an injective smooth map
$$c^{\infty}\otimes \mathbb C^2 \times \{\chi\in C^{\infty}(S^1,\mathbb C):\chi_0=0\} \times c^{\infty}\otimes \mathbb C^2 \to N^-N^+$$
with dense image; in particular we are asserting that $g$ has a triangular factorization $g=lu$.

(b) $g^{-*}$ has triangular factorization if and only if $1+\eta_i^-\eta_i^+\ne 0$ for all $i$, and $1+\zeta_k^-\zeta_k^+\ne 0$ for all $k$.

(b) Let $l=\left(\begin{matrix}\alpha_1&\beta_1\\\gamma_1&\delta_1\end{matrix}\right)$, $u=\left(\begin{matrix}\alpha_2&\beta_2\\\gamma_2&\delta_2\end{matrix}\right)$, and similarly for $g^{-*}=l'd'u'$ (which exists generically). $\eta_i^{\pm}$ are rational functions of the first $i$ coefficients of the Taylor expansions of $\gamma_1/\alpha_1$ and $\gamma'_1/\alpha'_1$ at $z=\infty$, and $\zeta_k^{\pm}$ are rational functions of the first $k$ coefficients of the Taylor expansions of $\gamma_2/\delta_2$ and $\gamma'_2/\delta'_2$ at $z=0$.
\end{theorem}

\subsection{Toeplitz Determinants}

\begin{theorem} Suppose that
$$g=\left(\prod_{i\ge 0}^{\rightarrow} \left(\begin{matrix}1+\eta_i^-\eta_i^+ & \eta^+ z^{i}\\ \eta^- z^{-i}& 1\end{matrix}\right)\right)^*\left(\begin{matrix}e^{\chi}&0\\0&e^{-\chi}\end{matrix}\right)
\prod_{k>0}^{\leftarrow}\left(\begin{matrix} 1+\zeta^-_k\zeta^+_k&\zeta^-_kz^{-k}\\
\zeta^+_kz^k&1\end{matrix} \right) $$
Then
$$det(A(g)A(g^{-1}))=\left(\prod_{i=0}^{\infty}(1+\eta^-_i\eta^+_i)^{i}\right)^*\times
\left(\prod_{
j=1}^{\infty}e^{2j\chi_j\chi_{-j}}\right)\times
\left(\prod_{k=1}^{\infty}(1+\zeta^-_k\zeta^+_{-k})^{k}\right)$$

\end{theorem}

\end{document}